\documentclass[a4paper, 11pt]{amsart}
\usepackage[T1]{fontenc}
\usepackage[utf8]{inputenc}
\usepackage{amsmath}
\usepackage{amsfonts}
\usepackage{amssymb}
\usepackage{amsthm}
\usepackage{graphicx}
\usepackage[english]{babel}
\usepackage{version} 
\usepackage{nicefrac}
\usepackage{tipa} 
\usepackage{hyperref}

\usepackage{xcolor}
\usepackage{soul}
\theoremstyle{definition}
\newtheorem{defin}{Definition}[section]

\theoremstyle{plain}

\newtheorem{lemma}[defin]{Lemma}
\newtheorem{obs}[defin]{Remark}
\newtheorem{prop}[defin]{Proposition}

\newtheorem*{theo-intro}{Theorem}
\newtheorem{theorem}{Theorem}

\newtheorem*{example}{Example}

\newcommand{\id}{\textup{id}}  
\newcommand{\Is}{\textup{Isom}}  

\newcommand{\R}{\mathbb{R}}  
\newcommand{\Z}{\mathbb{Z}} 
\newcommand{\N}{\mathbb{N}} 
 
\newcommand{\Eucl}{\textup{Eucl-rk}}
\newcommand{\Ab}{\textup{Ab-rk}}
\newcommand{\rk}{\textup{rk}}
\newcommand{\CAT}{\textup{CAT}}
\newcommand{\CATtdu}{ \textup{\normalfont CAT}_0^{\textup{\normalfont td-u}}(P_0,r_0,D_0)}

\renewenvironment{abstract}
{\par\noindent\textbf{\abstractname.}\ \ignorespaces}
{\par\medskip}

\title[Euclidean factor of CAT$(0)$-spaces]{GH-convergence of CAT$(0)$-spaces: stability of the Euclidean factor}
\author{Nicola Cavallucci}
\address{Nicola Cavallucci, EPFL FSB SMA, Station 8, 1015 Lausanne}
\email{n.cavallucci23@gmail.com}
\date{}

\begin{document}
	\maketitle
	
	\footnotesize
	\begin{abstract}
		We prove that if a sequence of geodesically complete CAT$(0)$-spaces $X_j$ with uniformly cocompact discrete groups of isometries converges in the Gromov-Hausdorff sense to $X_\infty$, then the dimension of the maximal Euclidean factor splitted off by $X_\infty$ and $X_j$ is the same, for $j$ big enough. In other words, no additional Euclidean factors can appear in the limit.
	\end{abstract}
	\normalsize
	
	\tableofcontents

	\section{Introduction}
	
	Given a CAT$(0)$-space $X$ we denote by $\text{Eucl-rk}(X)$ the dimension of the maximal Euclidean factor of $X$, i.e. the maximal $k$ for which $X$ splits metrically as $Y \times \R^k$: it is called the \emph{Euclidean rank} of $X$.
	The aim of this paper is to prove the following stability result.
	\begin{theorem}
		\label{theo-intro-stability}
		Let $X_j$ be a sequence of proper, geodesically complete, \textup{CAT}$(0)$-spaces with discrete and $D_0$-cocompact groups $\Gamma_j < \Is(X_j)$. Suppose $X_j$ converges in the pointed Gromov-Hausdorff sense to $X_\infty$. Then
		$$\Eucl(X_\infty) = \lim_{j \to + \infty} \Eucl(X_j).$$
	\end{theorem}
	The pointed Gromov-Hausdorff convergence needs basepoints on the spaces $X_j$ and $X_\infty$ in order to be defined. However, under the assumptions above, the limit $X_\infty$ does not depend on the choice of the basepoints, see Section \ref{subsec-basic}. Theorem \ref{theo-intro-stability} is simpler if we assume that each $G_j$ is \emph{torsion-free}.	Under this assumption, and more generally under the \emph{nonsingularity} of the $\Gamma_j$'s, it was proved true by the author and A.Sambusetti in \cite[Corollary 7.17]{CS23}. \\
	We will show a stronger result.
	\begin{theorem}
		\label{theo-intro-td}
		Let $X_j$ be a sequence of proper, geodesically complete, \textup{CAT}$(0)$-spaces with closed, totally disconnected, unimodular and $D_0$-cocompact groups $G_j < \textup{Isom}(X_j)$. If $X_j$ converges in the pointed Gromov-Hausdorff sense to $X_\infty$ then
		$$\Eucl(X_\infty) = \lim_{j \to + \infty} \Eucl(X_j).$$
	\end{theorem}
	
	Theorem \ref{theo-intro-td} will be consequence of the precise description of the behaviour of the groups $G_j$ as in the statement given in \cite{Cav23-GHTD}. We remark that the difficulty in proving Theorem \ref{theo-intro-stability} and Theorem \ref{theo-intro-td} is substantially the same. The motivations for that are well described in \cite{Cav23-GHTD}, in which it is shown that, even if the groups $G_j$'s are discrete, the limit  $G_\infty < \Is(X_\infty)$ is in general non-discrete. However it is always totally disconnected, provided the sequence $G_j$ is non-collapsed. In \cite{Cav23-GHTD} it is proved that every sequence $G_j$ as above can be transformed in a non-collapsed one without changing the isometry type of the spaces $X_j$, we refer to Section \ref{subsec-basic-convergence} for the details. Another motivation for studying the totally disconnected, unimodular case is provided by the decomposition of \cite{CM09b}, where these groups appear naturally. We will come back to this at the end of the introduction.\\
	Our proof of Theorem \ref{theo-intro-td} will follow the same scheme of the proof of Theorem \ref{theo-intro-stability} in the torsion-free case given in \cite{CS23}, overcoming the complications resulting from the total disconnection of the groups. One of the missing ingredients is an analogue, in our situation, of \cite[Theorem 2]{CM19}, which characterizes the Euclidean rank of a proper CAT$(0)$-space with a uniform lattice $\Gamma$ as the maximal rank of a free abelian commensurated subgroup of $\Gamma$. We define the \emph{abelian rank} of a locally compact group $G$ as the maximal rank of an almost abelian, almost commensurated subgroup and we denote it by $\Ab(G)$. The notions of almost abelianity and almost commensurability has been introduced in \cite{Cav23-GHTD}, see also Section \ref{subsec-basic-almost-abelian} for more details. The generalization of \cite[Theorem 2]{CM19} is the following.
	\begin{theorem}
		\label{theo-intro-characterization}
		Let $X$ be a proper, geodesically complete, \textup{CAT}$(0)$-space and let $G<\textup{Isom}(X)$ be closed, totally disconnected, unimodular and cocompact. Then $\Eucl(X)=\Ab(G)$.
	\end{theorem}
	
	This allows to approach Theorem \ref{theo-intro-td} by using the abelian rank. Ultimately what we will show is the next result.
	\begin{theorem}
		\label{theo-intro-stability-abelian-rank}
		Let $X_j$ be a sequence of proper, geodesically complete, \textup{CAT}$(0)$-spaces with closed, totally disconnected, unimodular and $D_0$-cocompact groups $G_j < \textup{Isom}(X_j)$. Suppose $(X_j. G_j)$ converges in the equivariant pointed Gromov-Hausdorff sense to $(X_\infty, G_\infty)$ with $G_\infty$ totally disconnected. Then
		$$\Ab(G_\infty) = \lim_{j\to + \infty} \Ab(G_j).$$
	\end{theorem}
	
	A natural question is the following: what happens if we have a sequence of proper, geodesically complete, $\CAT(0)$-spaces $X_j$ converging to $X_\infty$ and we only know that $\Is(X_j)$ is $D_0$-cocompact and unimodular for every $j$? 
	Here the unimodularity assumption of $\Is(X_j)$ is very natural and it is satisfied as soon as $\Is(X_j)$ possesses a lattice. The next example shows that in general the limit space $X_\infty$ can split more Euclidean factors.
	\begin{example}
		Let $M$ be a symmetric space of non-compact type equipped with a fixed Riemannian metric $g$ of non-positive curvature. Notice that $\Is(M)$ is unimodular and transitive, so $0$-cocompact. Consider now the metric $\lambda g$ on $M$ with $\lambda \geq 1$: it is again a geodesically complete $\CAT(0)$-space whose isometry group is unimodular and $0$-cocompact. For $\lambda$ going to $+\infty$ this sequence of spaces converges in the pointed Gromov-Hausdorff sense to $\R^n$, where $n$ is the dimension of $M$.\\
		Observe that there are no lattices with uniform codiameter along the spaces of this sequence, so the assumptions of Theorem \ref{theo-intro-stability} are not satisfied.
	\end{example}
	
	However, using the decomposition results of \cite{CM09b} we also cited before, we can prove that the above example is in some sense the only thing that can happen.
	
	\begin{theorem}
		\label{theo-intro-limit-whole-isometries}
		Let $X_j$ be a sequence of proper, geodesically complete, \textup{CAT}$(0)$-spaces with $\Is(X_j)$ unimodular and $D_0$-cocompact. Suppose $X_j$ converges in the pointed Gromov-Hausdorff sense to $X_\infty$. Then each $X_j$ splits as $M_j \times \R^{k_j} \times N_j$, where $M_j$ is a symmetric space of non-compact type, $k_j = \Eucl(X_j)$ and $\Is(N_j)$ is totally disconnected and unimodular. Moreover for every subsequence $\lbrace j_k \rbrace$ for which $M_{j_h}$ converges to a space $M_\infty$ we have
			$$\Eucl(X_\infty) = \Eucl(M_\infty) + \lim_{h \to + \infty} \Eucl(X_{j_h}).$$
	\end{theorem}
	
	We end the introduction with a classical rigidity application of stability results as Theorem \ref{theo-intro-td}, when coupled with some form of compactness.
	\begin{theorem}
		\label{theo-intro-rigidity-packed}
		Let $P_0,r_0,D_0 > 0$. Then there exists $\varepsilon = \varepsilon(P_0,r_0,D_0) > 0$ such that the following holds true. Let $X, Y$ be proper, geodesically complete, $(P_0,r_0)$-packed, $\CAT(0)$-spaces with closed, totally disconnected, unimodular, $D_0$-cocompact groups $G<\Is(X)$ and $H<\Is(Y)$. Let $x\in X$ and $y\in Y$ be points. If $d_{\textup{GH}}(\overline{B}(x,1/\varepsilon), \overline{B}(y,1/\varepsilon)) \leq \varepsilon$ then $\Eucl(X) = \Eucl(Y)$.
	\end{theorem}
	Here $d_{\textup{GH}}$ denotes the Gromov-Hausdorff distance on the class of compact metric spaces up to isometry. Essentially if two space as in the statement have balls \emph{of fixed radius} which are similar \emph{up to a fixed error} then they must split the same Euclidean factor. This is an istance of a general principle: the geometric information of spaces as in the statement of Theorem \ref{theo-intro-rigidity-packed} should be detected \emph{looking at a fixed scale}. We will not pursue this naive idea further in this paper.	
	
 \section{Notation and recalls}
 We recall the essential tools we will need throughout the paper.
 
 \subsection{$\CAT(0)$-spaces, packing and isometries}
 \label{subsec-basic}
 In the paper $X$ will denote a {\em proper} metric space with distance $d$. 
 A {\em geodesic} in $X$ is an isometry $c\colon [a,b] \to X$, where $[a,b]$ is an interval of $\mathbb{R}$. The {\em endpoints} of the geodesic $c$ are the points $c(a)$ and $c(b)$; a geodesic with endpoints $x,y\in X$ is also denoted by $[x,y]$. 
 A \emph{geodesic line} is an isometry $c\colon \mathbb{R} \to X$. 
 A metric space $X$ is called   {\em geodesic}  if for every two points $x,y \in X$ there is a geodesic with endpoints $x$ and $y$. 
 
 \noindent A metric space $X$ is called CAT$(0)$ if it is geodesic and every geodesic triangle $\Delta(x,y,z)$  is thinner than its Euclidean comparison triangle   $\overline{\Delta} (\bar{x},\bar{y},\bar{z})$: that is,  for any couple of points $p\in [x,y]$ and $q\in [x,z]$ we have $d(p,q)\leq d(\bar{p},\bar{q})$ where $\bar{p},\bar{q}$ are the corresponding points in $\overline{\Delta} (\bar{x},\bar{y},\bar{z})$ (see   for instance \cite{BH09} for the basics of CAT$(0)$-geometry).
 A CAT$(0)$-space is {\em uniquely geodesic}: for every $x,y \in X$ there exists a unique geodesic with endpoints $x$ and $y$.

 \noindent A CAT$(0)$-metric space $X$ is {\em geodesically complete} if every geodesic can be extended to a geodesic line.
 \vspace{2mm}
 
 Let $\text{Isom}(X)$ be the group of isometries of $X$, endowed with the compact-open topology: as $X$ is proper, it is a topological, locally compact, $\sigma$-compact  group. 
 \noindent A subgroup $G$ is \emph{totally disconnected} if it is totally disconnected as a subset of \textup{Isom}$(X)$ (with respect to the compact-open topology). 
 A closed group $G < \textup{Isom}(X)$ is said to be {\em cocompact} if the quotient metric space $G \backslash X$ is compact; in this case, we call \emph{codiameter} of $G$ the diameter of the quotient, and we will say that $G$ is $D_0$-cocompact if it has codiameter at most $D_0$.
 \noindent The {\em translation length} of $g\in \text{Isom}(X)$ is by definition  $\ell(g) := \inf_{x\in X}d(x,gx).$ 
 When the infimum is realized, the isometry $g$ is called {\em elliptic} if $\ell(g) = 0$ and {\em hyperbolic} otherwise. The {\em free-systole} of a group $G < \textup{Isom}(X)$ {\em at a point} $x\in X$ is
 $$\text{sys}^\diamond(G,x) := \inf_{g\in G \setminus G^\diamond} d(x,gx),$$
 where $G^\diamond $ is the subset of all elliptic isometries of $G$. The {\em free-systole of} $G$ is accordingly defined as 
 $$\text{sys}^\diamond(G,X) = \inf_{x\in X}\text{sys}^\diamond(G,x).$$
 The next splitting theorem is due to P.E.Caprace and N.Monod.
 \begin{prop}
 	\label{prop-CM-decomposition}
 	Let $X$ be a proper, geodesically complete, $\CAT(0)$-space such that $\Is(X)$ is cocompact and unimodular. Then $X$ splits metrically as $M \times \R^k \times N$ in such a way that $\Is(X) = S \times \mathcal{E}_{k} \times D$, with $\Is(M) = S$, $\Is(\mathbb{R}^{k}) = \mathcal{E}_{k}$ and $\Is(N) = D$. Moreover
 	\begin{itemize}
 		\item[(a)] $S$ is an almost connected, semi-simple Lie group with no compact factor and trivial center, $M$ is a simply connected, non-positively curved, symmetric Riemannian manifold;
 		\item[(b)] $k = \Eucl(X)$;
 		\item[(c)] $D$ is locally compact, totally disconnected, unimodular and cocompact with codiameter at most the codiameter of $\Is(X)$.
 	\end{itemize}
 \end{prop}
 \begin{proof}
 	The action of $\Is(X)$ on $X$ is minimal, i.e. there are no proper, convex, closed, invariant subsets, since $X$ is geodesically complete and $\Is(X)$ is cocompact (cp. \cite[Proposition 1.5]{CM09b}). By \cite[Theorem M]{CM13-unimodularity} the fixed points at infinity of $\Is(X)$ are contained in the boundary of the Euclidean factor of $X$, by unimodularity. Therefore $\Is(X)$ has no fixed points at infinity since the isometries of the Euclidean factor have none. Hence \cite[Theorem 1.6 and Addendum 1.8]{CM09b} gives a canonical $\Is(X)$-invariant splitting of $X$ as $M \times \R^k \times N$, where $M$ and $S=\Is(M)$ are as in (a), $k=\Eucl(X)$ and $D=\Is(N)$ is totally disconnected. The fact that $\Is(X)$ splits as $S\times \mathcal{E}_k \times D$ follows for instance from \cite[Proposition I.5.3.(4)]{BH09} because the splitting is $\Is(X)$-invariant. $D$ is also cocompact with codiameter bounded by the codiameter of $\Is(X)$, being the projection of $\Is(X)$ on $\Is(N)$. 
 	By \cite[Corollary p.85]{Nac65} the product of locally compact groups is unimodular if and only if the factors are unimodular. Since $\Is(X)$ is unimodular we deduce that also $D$ is unimodular.
 \end{proof}

 \begin{obs}
 	\label{rmk-unimodularity-td-isom}
 	If $X$ is a space as in Proposition \ref{prop-CM-decomposition} then it is equivalent to require that $\Is(X)$ is cocompact and unimodular or that there exists $G<\Is(X)$ closed, totally disconnected, unimodular and cocompact. Indeed if $\Is(X)$ is unimodular and cocompact then the previous proposition applies and it is enough to take $G=\Gamma\times D$, where $\Gamma$ is a lattice in $S\times \mathcal{E}_k$. Notice however that it is not possible in general to find such $G$ with codiameter bounded in terms of the codiameter of $\Is(X)$. This follows by comparing the Example in the introduction with Theorem \ref{theo-intro-td}.\\ 	
 	Viceversa if there exists $G<\Is(X)$ which is closed, totally disconnected, unimodular and cocompact then $\Is(X)$ is unimodular because the quotient space $\Is(X)/G$ is compact and \cite[Corollary B.1.8]{BdLHV08} applies. The fact that $\Is(X)/G$ is compact can be proved as follows. Fix $x_0 \in X$, let $D$ be the codiameter of the action of $G$ on $X$ and define the compact set $\overline{S}_{D} = \lbrace h\in \Is(X) \text{ s.t. } d(x_0,hx_0)\leq D\rbrace$. For every $h\in \Is(X)$ we can find $g\in G$ such that $d(hx_0,gx_0)\leq D$, i.e. $g^{-1}h \in \overline{S}_D$. In other words $\Is(X) = G\cdot \overline{S}_D$, implying that $\Is(X)/G$ is compact.
 \end{obs}
  
 Let $X$ be a metric space and $r>0$.
 A subset $Y$ of $X$ is called {\em $r$-separated} if $d(y,y') > r$ for all $y,y'\in Y$.
 Given $x\in X$ and $0<r\leq R$ we denote by Pack$(\overline{B}(x,R), r)$ the maximal cardinality of a $2r$-separated subset of $\overline{B}(x,R)$. Moreover we denote by Pack$(R,r)$ the supremum of Pack$(\overline{B}(x,R), r)$ among all points of $X$. 	
 Given $P_0,r_0 > 0$ we say that $X$ is {\em $(P_0,r_0)$-packed} if Pack$(3r_0,r_0) \leq P_0$.
 The packing condition should be thought as a metric, weak replacement  of a Ricci curvature lower bound: for more details and examples see \cite{CavS20} and \cite{CavS20bis}. Moreover every metric space admitting a cocompact group of isometries is packed (for some $P_0, r_0$), see the proof of \cite[Lemma 5.4]{Cav21ter}. The packing condition has many interesting geometric consequences for complete, geodesically complete CAT$(0)$-spaces, as showed in \cite{CavS20bis}, \cite{Cav21bis} and  \cite{Cav21}.

 \subsection{Gromov-Hausdorff convergence}
 \label{subsec-basic-convergence}
 
 An {\em isometric action} on a pointed space is a triple  $(X,x,G)$ where $X$ is a proper metric space, $x \in X$ is a basepoint and $G < \text{Isom}(X)$ is a closed subgroup. An {\em equivariant isometry} between isometric  actions of pointed spaces  $(X,x,G)$ and $(Y,y,H)$ is an isometry $F\colon X \to Y$ such that
 \begin{itemize}
 	\item[--] $F(x)=y$;
 	\item[--] $F_*\colon \textup{Isom}(X) \to \textup{Isom}(Y)$ defined by $F_*(g) = F\circ g \circ F^{-1}$ is an isomorphism between $G$ and $H$.
 \end{itemize}
 \noindent The best known notion of convergence for isometric actions of pointed spaces is the  \emph{equivariant pointed Gromov-Hausdorff} convergence,  as defined by K. Fukaya in \cite{Fuk86}:
 we will write $(X_j,x_j,G_j) \underset{\textup{eq-pGH}}{\longrightarrow} (X_\infty,x_\infty,G_\infty)$  
 for a sequence $(X_j,x_j,G_j)$  of isometric actions converging in the  equivariant pointed Gromov-Hausdorff sense to an isometric action $(X_\infty,x_\infty,G_\infty)$. 
 
 \noindent Forgetting about the group actions and  considering just  pointed metric spaces $(X_j,x_j)$,  this convergence reduces to the \emph{pointed Gromov-Hausdorff} convergence: we will write $(X_j,x_j) \underset{\textup{pGH}}{\longrightarrow} (X_\infty,x_\infty)$ for a sequence of pointed metric spaces $(X_j,x_j)$ converging to the pointed metric space $(X_\infty, x_\infty)$  in this sense. 
 An equivalent approach uses ultralimits. We briefly recall it. A {\em non-principal ultrafilter} $\omega$ is a finitely additive measure on $\mathbb{N}$ such that $\omega(A) \in \lbrace 0,1 \rbrace$ for every $A\subseteq \mathbb{N}$ and $\omega(A)=0$ for every finite subset of $\mathbb{N}$. \linebreak Accordingly, we will write  {\em $\omega$-a.s.}  or {\em for $\omega$-a.e.$(j)$} in the usual measure theoretic sense.  
 Given a bounded sequence $(a_j)$ of real numbers and a non-principal ultrafilter $\omega$ there exists a unique $a\in \mathbb{R}$ such that for every $\varepsilon > 0$ the set $\lbrace j \in \mathbb{N} \text{ s.t. } \vert a_j - a \vert < \varepsilon\rbrace$ has $\omega$-measure $1$ (cp. \cite[Lemma 10.25]{DK18}). The real number $a$ is then called {\em the ultralimit of the sequence $a_j$} and it is denoted by $\omega$-$\lim a_j$.
 
 \noindent A non-principal ultrafilter  $\omega$ being given,   one defines the {\em ultralimit pointed metric space} 
 $(X_\omega, x_\omega)= \omega$-$\lim (X_j, x_j)$ of any  sequence of pointed metric spaces $(X_j, x_j)$: \\
 -- first, one  says that a sequence    $(y_j)$, where $y_j\in X_j$ for every $j$, is {\em admissible} if there exists $M$ such that $d(x_j,y_j)\leq M$ for $\omega$-a.e.$(j)$; \\
 -- then, one 
 defines $(X_\omega, x_\omega)$ as  set of admissible sequences $(y_j)$ modulo the relation $(y_j)\sim (y_j')$ if and only if $\omega$-$\lim d(y_j,y_j') = 0$. \\
 The point of $X_\omega$ defined by the class of the sequence $(y_j)$ is denoted by  $y_\omega = \omega$-$\lim y_j$. 
 Finally, the formula $d(\omega$-$\lim y_j, \omega$-$\lim y_j') = \omega$-$\lim d(y_j,y_j')$ defines a metric on $X_\omega$ which is called the ultralimit distance on $X_\omega$.
 
 \noindent Using a non-principal ultrafilter  $\omega$, one can also talk of limits of  isometries  and  of  isometry groups of pointed metric spaces. A sequence of isometries  $g_j $ of pointed metric spaces $(X_j, x_j)$ is {\em admissible} if there exists $M\geq 0$ such that $d(x_j, g_jx_j) \leq M$ $\omega$-a.s. Every such sequence defines a limit isometry $g_\omega = \omega$-$\lim g_j$ of $X_\omega=\omega$-$\lim (X_j, x_j)$  by the formula: $g_\omega y_\omega = \omega$-$\lim g_jy_j$  (\cite[Lemma 10.48]{DK18}).
 Given a sequence of groups $G_j < \text{Isom}(X_j)$ we set
 $$G_\omega = \lbrace \omega\text{-}\lim g_j \text{ s.t. } g_j \in G_j \text{ for } \omega\text{-a.e.}(j)\rbrace.$$
 In particular the elements of $G_\omega$ are ultralimits of admissible sequences. 
 One has a well-defined composition law on $G_\omega$ (\cite[Lemma 3.7]{Cav21ter}): if $ g_\omega   = \omega$-$\lim g_j$ and  $ h_\omega = \omega$-$\lim h_j$ we set $g_\omega \, \circ\, h_\omega := \omega\text{-}\lim(g_j \circ h_j).$
 With this operation $G_\omega$ becomes a group of isometries of $X_\omega$, which  we call  {\em the ultralimit group} of the sequence of groups $G_j$.  
 Notice that if $X_\omega$ is proper then $G_\omega$ is always a closed subgroup of isometries of $X_\omega$ \cite[Proposition 3.8]{Cav21ter}.
 
 \vspace{1mm}
 \noindent In conclusion,  a non-principal ultrafilter  $\omega$ being given, for any sequence of isometric actions on pointed  spaces $(X_j,x_j, G_j)$ there exists an {\em ultralimit isometric action} on a pointed space
 \vspace{-3mm}
 
 $$(X_\omega, x_\omega, G_\omega) =  \omega \text{-}\lim (X_j,x_j, G_j).$$ 
 
 \vspace{1mm}
 \noindent The ultralimit approach and the Gromov-Hausdorff convergence are essentially equivalent.
 
 \begin{prop}[\cite{Cav21ter}, Proposition 3.13 \& Corollary 3.14] 
 	\label{prop-GH-ultralimit} ${}$\\
 	Let   $(X_j, x_j, G_j)$ be a sequence of isometric actions of pointed spaces:
 	\begin{itemize} 
 		\item[(i)]   if $(X_j,x_j,G_j) \underset{\textup{eq-pGH}}{\longrightarrow} (X_\infty,x_\infty,G_\infty)$, then $(X_\omega, x_\omega, G_\omega) \cong (X_\infty,x_\infty,G_\infty)$  for every non-principal ultrafilter $\omega$;
 		\item[(ii)] reciprocally, if $\omega$  is a non-principal ultrafilter and   $(X_\omega, x_\omega)$ is proper, then    $(X_{j_k},x_{j_k},G_{j_k}) \underset{\textup{eq-pGH}}{\longrightarrow} (X_\omega,x_\omega,G_\omega)$ for some subsequence $\lbrace{j_k}\rbrace$.
 	\end{itemize}
 	Moreover,  if for every non-principal ultrafilter $\omega$ the ultralimit  $(X_\omega,x_\omega,G_\omega)$ is equivariantly isometric to the same isometric  action   $(X,x,G)$, with $X$ proper,  then $(X_j,x_j,G_j) \underset{\textup{eq-pGH}}{\longrightarrow} (X,x,G)$.
 \end{prop}
 
 The sequence $(X_j,x_j, G_j)$   is called {\em $D$-cocompact} if  each $G_j$ is $D$-cocompact. 
 The ultralimit of a sequence of isometric  actions on pointed spaces does not depend on the choice of the basepoints, 
 provided that the actions have uniformly bounded codiameter (cp. \cite[Lemma 7.3]{CS23}). Therefore, when considering the convergence of uniformily cocompact isometric actions, 
 we will often omit the basepoints.
 \vspace{2mm}
 
 \noindent Given three parameters $P_0,r_0,D_0 > 0$ we denote by $\CATtdu$ the class of triples $(X,x,G)$ such that:
 \begin{itemize}
 	\item[-] $X$ is a proper, geodesically complete, $(P_0,r_0)$-packed, $\CAT(0)$-space;
 	\item[-] $x\in X$ is a basepoint;
 	\item[-] $G < \Is(X)$ is a closed, totally disconnected, unimodular, $D_0$-cocompact group.	
 \end{itemize}
 We also set $\textup{CAT}_0^{\textup{td-u}}(D_0) = \bigcup_{P_0,r_0 > 0} \CATtdu$.
 We can always assume to have a uniform packing, as soon as we have pointed Gromov-Hausdorff convergence.
 
 \begin{prop}[\textup{\cite[Proposition 7.5]{CS23}}]
 	\label{prop-GH-compactness-packing}
 	Suppose to have $X_j \in \textup{CAT}_0^{\textup{td-u}}(D_0)$ such that $X_j \underset{\textup{pGH}}{\longrightarrow} X_\infty$. Then $X_j \subseteq \CATtdu$ for every $j \in \N \cup \lbrace \infty \rbrace$, for some $P_0,r_0 > 0$. Moreover there exists $G_\infty < \Is(X_\infty)$ such that, up to passing to a subsequence, $(X_j,G_j) \underset{\textup{eq-pGH}}{\longrightarrow} (X_\infty, G_\infty)$.
 \end{prop} 
 
 \begin{defin}[Standard setting of convergence]\label{defsetting}${}$\\
 	We say that we are in the \emph{standard setting of convergence} when we have a sequence 
 	$(X_j,G_j)$
 	in $ \textup{CAT}_0^{\textup{td-u}}(P_0,r_0,D_0)$ such that 
 	$(X_j,G_j) \underset{\textup{eq-pGH}}{\longrightarrow} (X_\infty,G_\infty)$.\\
 	The standard setting of convergence (or simply the sequence) is called:
 	\begin{itemize}
 		\item[-]  \emph{non-collapsed} if $\limsup_{j\to+\infty} \textup{sys}^\diamond(G_j,X_j) > 0$;  
 		\item[-] \emph{collapsed} if $\liminf_{j\to+\infty} \textup{sys}^\diamond(G_j,X_j) = 0$.
 	\end{itemize}
 \end{defin}
 
 Every collapsed sequence can be transformed in a non-collapsed one without changing the isometry type of the spaces.
 \begin{prop}[\textup{\cite[Theorem 5.1]{Cav23-GHTD}}]
 	\label{prop-always-noncollapsed}
 	Let $P_0,r_0,D_0 > 0$. Then there exists $\Delta_0 > 0$ such that the following holds true. Suppose $(X_j,G_j) \underset{\textup{eq-pGH}}{\longrightarrow} (X_\infty,G_\infty)$ with $(X_j,G_j)$
 	in $\textup{CAT}_0^{\textup{td-u}}(P_0,r_0,D_0)$. Then there exist $G_j'<\Is(X_j)$, $j\in \N \cup \lbrace \infty \rbrace$ such that $(X_j, G_j') \in \textup{CAT}_0^{\textup{td-u}}(P_0,r_0,\Delta_0)$ and \linebreak $(X_j,G_j') \underset{\textup{eq-pGH}}{\longrightarrow} (X_\infty,G_\infty')$ without collapsing.
 \end{prop}

 By the characterization of collapsing sequences given in \cite{Cav23-GHTD} we have the following.
 \begin{prop}[\textup{\cite[Theorems 6.21 \& 6.22]{Cav23-GHTD}}]
 	\label{prop-limittd-iff-noncollapsed}
 	In the standard setting of convergence we have that $G_\infty$ is totally disconnected if and only if the sequence is non-collapsed. Moreover if it is the case then it is also unimodular.
 \end{prop}

Another important feature of non-collapsed sequences is the compactness of subgroups generated by small isometries.
\begin{prop}[\textup{\cite[Corollary 6.13 and §6.4]{Cav23-GHTD}}]
	\label{prop-non-collapsed-basepoints}
	In the non-collapsed standard setting of convergence there exists $\sigma > 0$ such that for every $y_j \in X_j$ the group generated by $\lbrace g_j \in G_j \text{ s.t. } d(y_j,g_jy_j) \leq \sigma \rbrace$ is compact.
\end{prop}
 
We end this section by explicitating the key lemma of the convergence theory developed in \cite{CS23} and \cite{Cav23-GHTD}. It is a deep result, despite its proof is a straightforward consequence of \cite{Cav23-GHTD}. A direct proof of it without the machinery developed there is out of reach for the author.
\begin{lemma}
	\label{lemma-key-elliptic-limit}
	Suppose to be in the standard setting of convergence. Let $g_j \in G_j$ be a sequence of admissible isometries defining the limit isometry $g_\infty \in G_\infty$. If $g_j$ is elliptic for every $j$ then $g_\infty$ is elliptic.
\end{lemma}
\begin{proof}
	By Proposition \ref{prop-GH-ultralimit} we can fix a non-principal ultrafilter $\omega$ and work with the ultralimit of the sequence $(X_j,x_j,G_j)$.
	The admissible condition means that $d(g_jx_j,x_j) \leq M$ for $\omega$-a.e.$(j)$ for some finite $M$. Let $y_j$ be the projection of $x_j$ on $\text{Fix}(g_j)$. By \cite[Corollary 6.17.(ii)]{Cav23-GHTD} we have $d(x_j,y_j) \leq N_0\cdot M$ for $\omega$-a.e.$(j)$, where $N_0$ is a constant depending only on $P_0,r_0$ and $D_0$. This means that the sequence $(y_j)$ is admissible and defines a limit point $y_\omega$ which is fixed by $g_\omega$, so $g_\omega$ is elliptic.
\end{proof}
 
 \subsection{Almost abelian groups}
 \label{subsec-basic-almost-abelian}
 In this part we recall the definitions of almost abelian and almost commensurated subgroup, as introduced in \cite{Cav23-GHTD}.
 \begin{defin}
 	A locally compact group $G$ is \emph{almost abelian} if there exists a compact, open, normal subgroup $N \triangleleft G$ such that $G/N$ is discrete, finitely generated and virtually abelian. 
 \end{defin}
 For instance if a locally compact group $G$ has a continuous surjective homomorphism $G \to \mathbb{Z}^k$ with compact kernel then it is almost abelian.
 Every almost abelian group has a well defined rank.
 \begin{defin}[\textup{\cite[Lemma 3.2]{Cav23-GHTD}}] Let $G$ be an almost abelian group. The rank of $G$, denoted by $\textup{rk}(G)$, is the unique integer $k\in \mathbb{N}$ such that the group $G/N$ is discrete, finitely generated and virtually abelian of rank $k$, for every compact, open, normal subgroup $N \triangleleft G$.
 \end{defin}
 
 \noindent We now present an example of almost abelian groups. We will use this criterion in the proof of Theorem \ref{theo-intro-stability-abelian-rank}.
 \begin{lemma}
 	\label{lemma-almost-commutator}
 	Let $A$ be a locally compact, compactly generated, totally disconnected group such that $\overline{[A,A]}$ is compact. Then $A$ is almost abelian.
 \end{lemma}
\begin{proof}
	The group $\overline{[A,A]}$ is compact and normal in $A$, so $A/\overline{[A,A]}$ is a locally compact, compactly generated, totally disconnected, abelian group. By the classification of such groups (cp. \cite[Theorem 24, p.85]{Mor77}) it is topologically isomorphic to $\mathbb{Z}^k \times K$ for some $k\geq 0$ and $K$ compact. The subgroup $\lbrace \id \rbrace \times K$ is open, compact and normal in $A/\overline{[A,A]}$. Let $\hat{K}$ be the preimage of $K$ under the projection map $\pi:A\to A/\overline{[A,A]}$. Then $\hat{K}$ is open in $A$ because $\pi$ is continuous. It is also compact because the kernel of $\pi$ is compact. Finally it is normal in $A$. Since $A/\hat{K}$ is topologically isomorphic to $\mathbb{Z}^k$, then $A$ is almost abelian.
\end{proof}
 
 We recall that a subgroup $H$ of a topological group $G$ is cocompact if the space $G/H$ is compact when endowed with the quotient topology.
 Let $G$ be a topological group and let $H < G$ be a closed subgroup. $H$ is said to be \emph{almost commensurated} if the group $H\cap gHg^{-1}$ is cocompact in both $H$ and $gHg^{-1}$, for all $g\in G$. We recall the following fact.
 \begin{lemma}[\textup{\cite[Lemma 4.2]{Cav23-GHTD}}]
 	\label{lemma-commensurability-generators}
 	Let $G$ be a topological group generated by a subset $S$. Let $H < G$ be a closed subgroup. Suppose $H\cap sHs^{-1}$ is cocompact in both $H$ and  $sHs^{-1}$ for every $s\in S$. Then $H$ is almost commensurated. 
 \end{lemma}
 Almost abelian almost commensurated subgroups of isometries of CAT$(0)$-groups induce a splitting of the space in the following sense.
 \begin{prop}[\textup{\cite[Proposition 4.5]{Cav23-GHTD}}]
 	\label{prop-characterization-inequality}
 	Let $X$ be a proper, geodesically complete, \textup{CAT}$(0)$-space and let $G<\Is(X)$ be closed, totally disconnected and cocompact. Let $A<G$ be a closed, almost abelian, almost commensurated subgroup of rank $k$. Then $X$ splits isometrically and $G$-invariantly as $X= Y \times \R^k$.
 \end{prop}
 Moreover almost abelian groups act as lattices on some convex subsets.
 \begin{prop}[\textup{\cite[Proposition 3.12]{Cav23-GHTD}}]
 	\label{prop-almost-abelian-lattice}
 	Let $X$ be a proper, geodesically complete, $\CAT(0)$-space and let $G<\Is(X)$ be closed, totally disconnected and cocompact. Let $A<G$ be an almost abelian subgroup of rank $k$. Then there exists a closed, convex, $A$-invariant subset $Y\subset X$ isometric to $\R^k$ such that $A$ acts a lattice on $Y$.
 \end{prop}

\subsection{Lattices in Euclidean spaces}
Euclidean spaces play a special role in the class of our $\CAT(0)$-spaces, as well highlighted in \cite{CS23} and \cite{Cav23-GHTD}.
We will denote points in $\mathbb{R}^k$ by a bold letter {\bf v} and the origin with ${\bf O}$. A lattice of $\mathbb{R}^k$ is a discrete, free abelian, cocompact group of isometries of $\mathbb{R}^k$.
It is well known that a lattice must act by translations on $\mathbb{R}^k$ (see for instance \cite{farkas}); so, alternatively, a lattice $\mathcal{L}$ can be seen as the set of linear combinations with integer coefficients of $k$ independent vectors $\bf b_1, \ldots, \bf b_k$. A {\em basis} $\mathcal{B} = \lbrace \bf b_1,\ldots,\bf b_k\rbrace$ of a lattice $\mathcal{L}$ is a set of $k$ independent vectors that generate $\mathcal{L}$ as a group. Given a basis $\mathcal{B}$ of a lattice $\mathcal{L}$ and an element ${\bf v} \in \mathcal{L}$ we denote by $\ell_{\mathcal{B}}({\bf v})$ the classical word length of ${\bf v}$ with respect to the generating set $\mathcal{B}$. We also set $\lambda(\mathcal{B}) := \max_i \Vert {\bf b_i} \Vert$.\\
There are many geometric invariants associated to a lattice $\mathcal{L}$, for instance:\\
-- the \emph{shortest generating radius}, that is
$$\lambda(\mathcal{L}) = \inf\lbrace r > 0 \text{ s.t. } \mathcal{L} \text{ contains } k \text{ independent vectors of length} \leq r\rbrace;$$
-- the \emph{shortest vector}, that is 
$$\tau(\mathcal{L}) = \inf \left\lbrace \Vert {\bf v} \Vert \text{ s.t. } {\bf v} \in \mathcal{L} \setminus \lbrace {\bf O} \rbrace \right\rbrace.$$
Notice that $\tau(\mathcal{L})$ coincides with the free-systole of $\mathcal{L}$, while the codiameter of $\mathcal{L}$ is at most $\sqrt{k}\cdot\lambda(\mathcal{L})$ (cp. for instance \cite[§2.3]{CS23}). The notation is coherent in the sense that $\lambda(\mathcal{L}) = \min_{\mathcal{B} \text{ basis of } \mathcal{L}} \lambda(\mathcal{B})$.

%
\vspace{1mm}
\noindent In the sequel we will need the next uniform estimate.

\begin{prop}
	\label{prop-lattices-uniform-word-length}
	Let $\tau, \lambda > 0$, $k\in \mathbb{N}$ and $R \geq 0$. There exists a constant $C = C(\tau,\lambda,k,R)$ such that the following holds true. Let $\mathcal{L}$ e a lattice of $\mathbb{R}^k$ with $\tau(\mathcal{L}) \geq \tau$. Let $\mathcal{B}$ be a basis of $\mathcal{L}$ with $\lambda(\mathcal{B}) \leq \lambda$. Let ${\bf v} \in \mathcal{L}$ be an element with $\Vert {\bf v} \Vert \leq R$. Then $\ell_{\mathcal{B}}({\bf v}) \leq C$.
\end{prop}

\begin{proof}
	Suppose the thesis is not true and find lattices $\mathcal{L}_j$ of $\mathbb{R}^k$, all satisfying $\tau(\mathcal{L}_j) \geq \tau$, bases $\mathcal{B}_j = \lbrace {\bf b_{1,j}}, \ldots {\bf b_{k,j}} \rbrace$ of $\mathcal{L}_j$ with $\lambda(\mathcal{B}_j) \leq \lambda$ and vectors ${\bf v_j} \in \mathcal{L}_j$ with $\Vert {\bf v_j} \Vert \leq R$ such that $\ell_{\mathcal{B}_j}({\bf v_j}) \geq j$. As recalled above the groups $\mathcal{L}_j$ have all codiameter at most $\sqrt{k}\cdot \lambda$ and free-systole at least $\tau$. Fix a non-principal ultrafilter $\omega$. A very special case of \cite[Theorem 7.10.(ii)]{CS23} says that the limit group $\mathcal{L}_\omega$ is again a lattice of $\mathbb{R}^k$ and that the natural projection map $\pi_j\colon\mathcal{L}_\omega \to \mathcal{L}_j$, ${\bf v_\omega} = \omega$-$\lim {\bf v_j} \mapsto {\bf v_j}$ is an isomorphism of groups for $\omega$-a.e.$(j)$. The sequences ${\bf b_{i,j}}$ are admissible and define isometries ${\bf b_{i,\omega}} \in \mathcal{L}_\omega$ for $i=1,\ldots,k$. Since $\pi_j$ is an isomorphism then $\mathcal{B}_\omega = \lbrace {\bf b_{1,\omega}}, \ldots {\bf b_{k,\omega}} \rbrace$ is a basis of $\mathcal{L}_\omega$. The vector ${\bf v_\omega} := \omega$-$\lim {\bf v_j}$ is a well defined element of $\mathcal{L}_\omega$. Therefore it can be written as product of elements of $\mathcal{B}_\omega$, say ${\bf v_\omega} = \beta_1{\bf b_{1,\omega}} + \cdots + \beta_k{\bf b_{k,\omega}}$. Since $\pi_j$ is an isomorphism then ${\bf v_j} = \beta_1{\bf b_{1,j}} + \cdots + \beta_k{\bf b_{k,j}}$, i.e. $\ell_{\mathcal{B}_j}({\bf v_j}) \leq \sum_{i=1}^k \vert \beta_i \vert < j$ if $j$ is big enough, giving the contradiction.
\end{proof}

%

 \section{The proof of Theorem \ref{theo-intro-characterization}}
 
 In this section we assume that $X$ is a proper, geodesically complete, CAT$(0)$-space $X$ and $G<\text{Isom}(X)$ is closed, totally disconnected, unimodular and cocompact. Remark \ref{rmk-unimodularity-td-isom} says that $\Is(X)$ is unimodular and cocompact, so Proposition \ref{prop-CM-decomposition} applies. In particular $\Is(X) = S \times \mathcal{E}_k \times D$, where $S$ is a semi-simple Lie group with trivial center and no compact factor, $\mathcal{E}_k = \Is(\R^k)$ and $D$ is locally compact and totally disconnected. We generalize \cite[Lemma 6]{CM19} to the totally disconnected case.
 
 \begin{prop}
 	\label{prop-commensurated-product}
 	Let $S \times \mathcal{E}_k \times D$ with $S$ semi-simple Lie group with trivial center and no compact factor, $\mathcal{E}_k = \Is(\R^k)$ and $D$ locally compact and totally disconnected.
 	Then every closed, cocompact, totally disconnected subgroup $G < S \times \mathcal{E}_k \times D$ commensurates an almost abelian subgroup of $G$ of rank $k$.
 \end{prop}
 \begin{proof}
 	Let $Q < D$ be a compact open subgroup. The intersection
 	$G^* = G \cap (S \times \mathcal{E}_k \times Q)$ is closed and totally disconnected. Indeed $G^*$ is a closed subgroup of $G$, so its connected component $(G^*)^o$ is a closed, connected subgroup of $G$, therefore it must be trivial. Moreover $G^*$ is cocompact in $S \times \mathcal{E}_k \times Q$. In order to prove it, we consider the continuous and open projection map $p\colon S\times \mathcal{E}_k \times D \to (S\times \mathcal{E}_k \times D) / G$. The set $p(S\times \mathcal{E}_k \times Q)$ is closed in the compact space $(S\times \mathcal{E}_k \times D) / G$, so compact as well. Moreover if $g_1,g_2 \in S\times \mathcal{E}_k \times Q$ are such that $p(g_1)=p(g_2)$ then $g_2^{-1}g_1 \in G \cap (S\times \mathcal{E}_k \times Q) = G^*$. Therefore the induced map $p\colon (S\times \mathcal{E}_k \times Q) / G^* \to (S\times \mathcal{E}_k \times D) / G$ is well defined, continuous, closed and injective. We deduce that $(S\times \mathcal{E}_k \times Q) / G^*$ is homeomorphic to its image through $p$ which is a compact space, showing that $G^*$ is cocompact in $ S\times \mathcal{E}_k \times Q$. Furthermore $G^*$ is commensurated by $G$. Indeed if $g=(g_S,g_{\mathcal{E}_k},g_D) \in G$ then
 	$$gG^*g^{-1} \cap G^* = G\cap (S\times \mathcal{E}_k \times (g_DQg_D^{-1}\cap Q))$$
 	and $g_DQg_D^{-1}\cap Q$ has finite index in $Q$ because it is an open subgroup of the compact group $Q$. So $gG^*g^{-1} \cap G^*$ has finite index in $G^*$ and in $gG^*g^{-1}$.\\
 	Since $Q$ is compact, the projection of $G^*$ to $S \times \mathcal{E}_k$ is closed, totally disconnected and cocompact. Indeed let $(g_{j,S}, g_{j,\mathcal{E}_k}, g_{j,Q})$ be a sequence of elements of $G^*$ such that $g_{j,S}$ converges to $g_{\infty,S} \in S$ and $g_{j,\mathcal{E}_k}$ converges to $g_{\infty, \mathcal{E}_k}$. By compactness of $Q$ we can suppose that also $g_{j,Q}$ converges to some $g_{\infty, Q}$, up to pass to a subsequence. Therefore the sequence $(g_{j,S}, g_{j,\mathcal{E}_k}, g_{j,Q})$ converges to $(g_{\infty,S}, g_{\infty,\mathcal{E}_k}, g_{\infty,Q}) \in G^*$ since $G^*$ is closed. This shows that $(g_{\infty,S}, g_{\infty,\mathcal{E}_k})$ belongs to $p_{S\times \mathcal{E}_k}(G^*)$, which is then closed. The cocompactness is clear.	The group $p_{S\times \mathcal{E}_k}(G^*)$ is therefore locally compact, being closed in $S\times \mathcal{E}_k$. Since $G^*$ is totally disconnected then it has a neighbourhood basis at the identity consisting of compact open subgroups. The map $p_{S\times \mathcal{E}_k}\colon G^* \to p_{S\times \mathcal{E}_k}( G^* )$ is a surjective, continuous homomorphism between locally compact groups, so it is open. This implies that also $p_{S\times \mathcal{E}_k}( G^* )$ has a neighbourhood basis at the identity consisting of compact open subgroups, so it is totally disconnected. The group $p_{S\times \mathcal{E}_k}(G^*)$ is a closed, totally disconnected subgroup of the Lie group $S\times A$, so it must be discrete. Since it is also cocompact it is a uniform lattice.
 	Upon replacing $G^*$ by a
 	finite index subgroup, we may then assume by \cite[Lemma 3.4 and 3.5]{CM09a} that
 	$G^*$ possesses two normal subgroups $G^*_S$ and $G^*_{\mathcal{E}_k}$, both commensurated by $G$, such that
 	$G^* = G^*_S \cdot G^*_{\mathcal{E}_k}$ and $G^*_S \cap G^*_{\mathcal{E}_k} \subseteq Q$, where $G^*_{\mathcal{E}_k} = G^* \cap (1 \times \mathcal{E}_k \times Q)$ is a group whose projection to $\mathcal{E}_k$ is a lattice. 
 	The group $G^*_{\mathcal{E}_k}$ is almost abelian of rank $k$, indeed it is a locally compact group with a surjective continuous homomorphism with values in $\mathbb{Z}^k$ whose kernel is compact.
 \end{proof}

 We are ready to provide the
 \begin{proof}[Proof of Theorem \ref{theo-intro-characterization}]
 As explained at the beginning of this section we know that $\Is(X) = S \times \mathcal{E}_k \times D$, where $k = \Eucl(X)$. The group $G$ is cocompact in $\Is(X)$ (see Remark \ref{rmk-unimodularity-td-isom}), so Proposition \ref{prop-commensurated-product} gives an almost abelian subgroup of $G$ with rank $k$ which is commensurated. This proves that $\Eucl(X) \leq \Ab(G)$. On the other hand Proposition \ref{prop-characterization-inequality} gives $\Ab(G) \leq \Eucl(X)$.
 \end{proof}
 
 \begin{obs}
 	Notice that the proof above gives something more: indeed it says that the abelian rank of a group $G$ as in the statement coincides with the maximal rank of an almost abelian group which is commensurated, which is stronger than being almost commensurated. However we will not use it.
 \end{obs}
	
\section{The proofs of Theorems \ref{theo-intro-stability-abelian-rank}, \ref{theo-intro-td} and \ref{theo-intro-rigidity-packed}}
	
We continue with the 
\begin{proof}[Proof of Theorem \ref{theo-intro-stability-abelian-rank}]
	We are in the standard setting of convergence with basepoints $x_j$.
	Since we are assuming that $G_\infty$ is totally disconnected then we are in the non-collapsed case, by Proposition \ref{prop-limittd-iff-noncollapsed}. Theorem \ref{theo-intro-characterization} says that $\Ab(G_j) = \Eucl(X_j)$ for every $j \in \mathbb{N} \cup \lbrace \infty \rbrace$. Moreover if infinitely many $X_j$'s split a $\R^k$ then also $X_\infty$ splits a $\R^k$. This means that $\Eucl(X_\infty) \geq \limsup_{j\to +\infty} \Eucl(X_j)$, and so $\Ab(G_\infty) \geq \limsup_{j\to +\infty} \Ab(G_j).$
	The difficult part is to show that
	$$\Ab(G_\infty) \leq \liminf_{j\to +\infty} \Ab(G_j).$$
	In order to do so we fix a non-principal ultrafilter $\omega$. By Proposition \ref{prop-GH-ultralimit} and \cite[Lemma 6.3]{Cav21ter} it is enough to establish that
	$$\Ab(G_\omega) \leq \omega\text{-}\lim \Ab(G_j).$$
	We set $k := \Ab(G_\omega)$ and we fix an almost abelian, almost commensurated subgroup $A < G_\omega$ of rank $k$. Let $Y \subset X_\omega$ be a closed, convex, $A$-invariant subset isometric to $\R^k$ on which $A$ acts as a lattice $\mathcal{L}$ as in Proposition \ref{prop-almost-abelian-lattice}. We fix elements $a_{1,\omega}, \ldots, a_{k,\omega} \in A$ forming a basis of $\mathcal{L}$. Each isometry $a_{i,\omega}$ is by definition the limit of admissible isometries $a_{i,j} \in G_j$, $i=1,\ldots,k$. 
	We set $A_j := \langle a_{1,j}, \ldots, a_{k,j} \rangle$ and we claim that $\overline{A_j}$ is an almost abelian subgroup of rank $k$ which is almost commensurated in $G_j$ for $\omega$-a.e.$(j)$. By definition this implies that $\Ab(G_j) \geq k$ for $\omega$-a.e.$(j)$ concluding the proof.\\
	\textbf{Step 1:} \emph{$\overline{[A_j,A_j]}$ is compact for $\omega$-a.e.$(j)$.} \\
	We fix a point $y_\omega = \omega$-$\lim y_j \in Y$. By assumption the commutators $[a_{i,\omega}, a_{\ell,\omega}]$ fix the point $y_\omega$ for $1\leq i,\ell \leq k$. So we have $d([a_{i,j}, a_{\ell,j}]y_j, y_j) \leq \sigma$ for all $1\leq i,\ell\leq k$, $\omega$-a.s. Choosing $\sigma$ as in Proposition \ref{prop-non-collapsed-basepoints} we have that all these commutators belong to the same compact subgroup of $G_j$, $\omega$-a.s. Therefore the closure of the group generated by them is compact.\\
	\textbf{Step 2:} \emph{$\overline{A_j}$ is almost abelian for $\omega$-a.e.$(j)$.}\\
	The proof follows by Lemma \ref{lemma-almost-commutator}, once we have shown that $\overline{A_j}$ is compactly generated and that $\overline{[\overline{A_j}, \overline{A_j}]}$ is compact. 
	The second condition is easier. Indeed by Step 1 we know that there exists a point $z_j\in X_j$ fixed by $[A_j,A_j]$. Now take two elements $g_j,h_j\in \overline{A_j}$, so each of them is limit of isometries of $A_j$. The commutator $[g_j,h_j]$ is limit of commutators of elements of $A_j$, so it must fix $z_j$ too. This shows that $[\overline{A_j}, \overline{A_j}]$ fixes the point $z_j$, so its closure is compact. \\
	It remains to show that $\overline{A_j}$ is compactly generated. We will prove that $\overline{A_j} = \langle A_j, \overline{A_j} \cap B_j \rangle$, where $B_j := \text{Stab}_{G_j}(x_j)$. This would suffice since the set $A_j \cup (\overline{A_j} \cap B_j)$ is compact. Let $w_{n,j}$ be a sequence of elements of $A_j$ indexed by $n$ and converging to some $w_j \in G_j$. Since the orbits of $G_j$ are discrete (cp. \cite[Theorem 2.6]{Cav23-GHTD}) we know that $w_{n,j} x_j = w_jx_j$ for all $n \geq n_0$, for some $n_0$ depending on $j$. In particular $w_{n_0,j}^{-1}w_{n,j} x_j = x_j$ for all $n\geq n_0$. This means that we can write $w_{n,j} = w_{n_0,j} \cdot b_{n,j}$ for all $n\geq n_0$, for some $b_{n,j} \in B_j$. Since the sequence $w_{n,j}$ converges to $w_j$ then the sequence $b_{n,j}$ converges to $b_j := w_{n_0,j}^{-1} w_{j}$. Therefore $b_j \in \overline{A_j} \cap B_j$. Since $w_j = w_{n_0,j}\cdot b_j$ we deduce that $w_j \in \langle A_j, \overline{A_j} \cap B_j \rangle$. This shows that $\overline{A_j} \subseteq \langle A_j, \overline{A_j} \cap B_j \rangle$. The other inclusion is trivial, so $\overline{A_j} = \langle A_j, \overline{A_j} \cap B_j \rangle$.\\
	\textbf{Step 3:} \emph{$\rk(\overline{A_j}) = k$ for $\omega$-a.e.$(j)$.}\\
	$\overline{A_j}$ is almost abelian by Step 2. First of all $\rk(\overline{A_j}) \leq k$ for every $j$. Indeed let $N_j \triangleleft \overline{A_j}$ be a compact, open, normal subgroup such that $\overline{A_j}/N_j$ is virtually abelian, finitely generated and discrete with rank $=\rk(\overline{A_j})$. The commutator group of $\overline{A_j}/N_j$ is finite, being the image under the natural projection map of the compact group $\overline{[\overline{A_j}, \overline{A_j}]}$. We quotient $\overline{A_j}/N_j$ by its commutator subgroup and we obtain a finitely generated, abelian group of rank $=\rk(\overline{A_j})$ on which $\overline{A_j}$ surjects via a continuous homomorphism. Taking only the non-torsion part of this abelian group we can suppose that $\overline{A_j}$ surjects via a continuous homomorphism onto $\Z^\ell$, $\ell=\rk(\overline{A_j})$. Since the image of $A_j$ through this homomorphism is dense then it must be the whole $\Z^\ell$. Hence $\Z^\ell$ can be generated by $k$ elements, so $\rk(\overline{A_j}) = \ell \leq k$.\\
	The quantities $\rk(\overline{A_j})$ belong to the finite set $\lbrace 0,\ldots,k\rbrace$, so there exists $0\leq r \leq k$ such that $\rk(\overline{A_j}) = r$ for $\omega$.a.e.$(j)$. Suppose $r < k$. By Proposition \ref{prop-almost-abelian-lattice} we can find a closed, convex, $\overline{A_j}$-invariant subset $Y_j$ of $X_j$ isometric to $\R^r$ on which $\overline{A_j}$ acts as a lattice $\mathcal{L}_j$, $\omega$-a.s. The actions of $A_j$ and $\overline{A_j}$ coincide on $Y_j$, since the latter is discrete. Without loss of generality we can suppose that the first $r$ elements $a_{1,j}, \ldots, a_{r,j}$ form a basis $\mathcal{A}_j$ of $\mathcal{L}_j$. Recalling that the sequence is non-collapsed, i.e. we have a uniform lower bound on the length of every hyperbolic isometry of all the $G_j$'s, we can find $\tau > 0$ such that $\tau(\mathcal{L}_j) \geq\tau$ for $\omega$-a.e.$(j)$. Moreover, since the isometries $a_{i,\omega}$ are well defined, we can find $\Lambda \in \R$ such that $\ell(a_{i,\omega}) \leq \Lambda$ for every $i=1,\ldots,k$. It is obvious that $\ell(a_{i,j}) \leq 2\Lambda$ for $i=1,\ldots,k$, for $\omega$-a.e.$(j)$. Moreover each $a_{i,j}$ attains its minimum on $Y_j$ by \cite[Proposition II.6.2]{BH09}. This means that $\lambda(\mathcal{A}_j) \leq 2\Lambda$ $\omega$-a.s. Since $r<k$ the following is true on $Y_j$: $a_{k,j} = \Pi_{i=1}^{r} a_{i,j}^{\alpha_{i,j}}$ for some $\alpha_{i,j}$, $\omega$-a.s. Now observe that also $a_{k,j}$ attains its minimum on $Y_j$, and so that $\Vert a_{k,j} \Vert \leq 2\Lambda$ on $Y_j \cong \R^r$. We are in position to apply Proposition \ref{prop-lattices-uniform-word-length}: there exists a constant $C$, that does not depend on $j$, such that $\ell_{\mathcal{A}_j}(a_{k,j}) \leq C$. This means in particular that we can choose the $\alpha_{i,j}$'s satisfying $\sum_{i=1}^r \vert \alpha_{i,j}\vert \leq C$ for $\omega$-a.e.$(j)$. By triangle inequality the sequence of isometries $h_j := a_{k,j}^{-1}\cdot\Pi_{i=1}^{r} a_{i,j}^{\alpha_{i,j}}$ is admissible and defines the limit isometry $h_\omega = a_{k,\omega}^{-1}\cdot\Pi_{i=1}^{r} a_{i,\omega}^{\alpha_{i,\omega}}$,  where $\alpha_{i,\omega} = \omega$-$\lim \alpha_{i,j}$. Moreover the $h_j$'s are all elliptic, so $h_\omega$ is elliptic by Lemma \ref{lemma-key-elliptic-limit}. But $h_\omega$ belongs to $A$, so it attains its minimum on the closed, convex, invariant set $Y \subseteq X_\omega$, again by \cite[Proposition II.6.2]{BH09}. Since $A$ acts as a lattice on $Y$, the only possibility is that $h_\omega = \id$ on $Y$. This means that $a_{k,\omega} = \Pi_{i=1}^{r} a_{i,\omega}^{\alpha_{i,\omega}}$ on $Y$, which is absurd since $\lbrace a_{1,\omega}, \ldots, a_{k,\omega} \rbrace$ generates a $k$-dimensional lattice on $Y$. This concludes the proof of Step 3.\\
	\textbf{Step 4:} \emph{$A_\omega := \omega$-$\lim \overline{A_j} < G_\omega$ is almost abelian of rank $k$ and almost commensurated.}\\
	In the proof of Step 2 we showed that every element $g_j \in \overline{A_j}$ can be written as a product of an element of $A_j$ and an element of the compact group $B_j = \text{Stab}_{G_j}(x_j)$, i.e. $g_j = \Pi_{i=1}^k a_{i,j}^{\alpha_{i,j}} \cdot b_j$ with $b_j \in B_j$. Let us now take an admissible sequence $g_j$ with $g_j \in \overline{A_j}$. By definition there exists a finite $M$ such that $d(x_j, g_jx_j) \leq M$ for $\omega$-a.e.$(j)$. Since $b_j$ fixes $x_j$ this means $d(x_j, \Pi_{i=1}^k a_{i,j}^{\alpha_{i,j}}x_j) \leq M$ for $\omega$-a.e.$(j)$. As in the proof of Step 3 take a subset $Y_j$ of $X_j$ isomorphic to $\R^k$ on which $\overline{A_j}$ acts as a lattice $\mathcal{L}_j$ with $\tau(\mathcal{L}_j) \geq \tau$ and basis $\mathcal{A}_j = \lbrace a_{1,j}, \ldots, a_{k,j} \rbrace$ satisfying $\lambda(\mathcal{A}_j) \leq 2\Lambda$, for $\omega$-a.e.$(j)$. The isometry $h_j := \Pi_{i=1}^k a_{i,j}^{\alpha_{i,j}}$ attains its minimum $\ell(h_j) \leq M$ on $Y_j$. Another application of Proposition \ref{prop-lattices-uniform-word-length} as in Step 3 gives a uniform constant $C$, not depending on $j$, such that $\sum_{i=1}^k \vert \alpha_{i,j} \vert \leq C$. Therefore the admissible sequence of isometries $h_j$ converges to the isometry $h_\omega = \Pi_{i=1}^k a_{i,\omega}^{\alpha_{i,\omega}} \in A$. On the other hand the admissible sequence of isometries $b_j \in B_j$ converges to an isometry $b_\omega$ belonging to the compact group $B = \text{Stab}_{G_\omega}(x_\omega)$. We have just proved that $g_\omega = h_\omega \cdot b_\omega$. By the arbitrariness of $g_\omega$ we deduce that $A\subseteq A_\omega \subseteq A \cdot B$, i.e. $A$ is a cocompact subgroup of $A_\omega$. Then $A_\omega$ is almost abelian of rank $k$ by \cite[Lemma 3.4]{Cav23-GHTD} which is almost commensurated by \cite[Lemma 4.3]{Cav23-GHTD}.\\	
	\textbf{Step 5:} \emph{$\overline{A_j}$ is almost commensurated for $\omega$-a.e.$(j)$.}\\
	We recall that the groups $G_j$ are generated by the elements that move the basepoint $x_j$ by at most $2D_0$, see \cite[§2.4]{Cav23-GHTD}. Therefore, by Lemma \ref{lemma-commensurability-generators}, it is enough to check that $g_j \overline{A_j} g_j^{-1} \cap \overline{A_j}$ is cocompact in both $\overline{A_j}$ and $g_j \overline{A_j} g_j^{-1}$ for every $g_j$ such that $d(x_j,g_jx_j)\leq 2D_0$. By \cite[Lemma 3.4]{Cav23-GHTD} this is equivalent to ask that $\rk(g_j \overline{A_j} g_j^{-1} \cap \overline{A_j}) = k$ for $\omega$-a.e.$(j)$.
	Suppose that for $\omega$-a.e.$(j)$ we can find $g_j$ as above such that $\rk(g_j \overline{A_j} g_j^{-1} \cap \overline{A_j}) < k$. The sequence $g_j$ is admissible and defines the limit isometry $g_\omega \in G_\omega$. Moreover by definition $\omega$-$\lim (g_j \overline{A_j} g_j^{-1} \cap \overline{A_j}) = g_\omega A_\omega g_\omega^{-1} \cap A_\omega$. Step $4$ says that $A' = g_\omega A_\omega g_\omega^{-1} \cap A_\omega$ is almost abelian of rank $k$.
	Once again we can find a closed, convex, $A'$-invariant subset of $X_\omega$ which is isomorphic to $\R^k$ and on which $A'$ acts as a lattice generated by elements of the form $g_\omega a_{i,\omega}' g_\omega^{-1}$ for some $a_{i,\omega}' \in A_\omega$. Each $a_{i,\omega}'$ is the limit of isometries $a_{i,j}' \in G_j$. We define $A_j' = \langle g_j a_{1,j}' g_j^{-1}, \ldots, g_j a_{k,j}' g_j^{-1} \rangle < g_j \overline{A_j} g_j^{-1} \cap \overline{A_j}$. The same proofs of the Steps $1,2$ and $3$ give that the closure of $A_j'$ is almost abelian of rank $k$, giving the contradiction which concludes the proof of the theorem.
\end{proof}

Theorem \ref{theo-intro-td} easily follows.
\begin{proof}[Proof of Theorem \ref{theo-intro-td}]
	First of all we can suppose, up to passing to a subsequence, to be in the standard setting of convergence by Proposition \ref{prop-GH-compactness-packing}. Moreover by Proposition \ref{prop-always-noncollapsed} we can change the action of the groups $G_j$ without changing the spaces $X_j$ in order to be in the non-collapsed standard setting of convergence. Observe that this change does not affect the thesis which involves only geometric properties of the spaces. By Proposition \ref{prop-limittd-iff-noncollapsed} the limit group $G_\infty$ is totally disconnected. Theorem \ref{theo-intro-characterization} says that $\Eucl(X_j) = \Ab(G_j)$ for every $j \in \N \cup \lbrace \infty \rbrace$. On the other hand Theorem \ref{theo-intro-stability-abelian-rank} guarantees that $\Ab(G_\infty) = \lim_{j\to +\infty} \Ab(G_j)$. Therefore we conclude that $\Eucl(X_\infty) = \lim_{j \to + \infty} \Eucl(X_j)$.
\end{proof}

Theorem \ref{theo-intro-stability} is a particular case of Theorem \ref{theo-intro-td}.

\begin{proof}[Proof of Theorem \ref{theo-intro-rigidity-packed}]
	If it is not the case we can find proper, geodesically complete, $(P_0,r_0)$-packed spaces $X_j, Y_j$ with closed, totally disconnected, unimodular, $D_0$-cocompact groups $G_j < \Is(X_j)$ and $H_j < \Is(Y_j)$ and points $x_j \in X_j, y_j \in Y_j$ such that $d_{\textup{GH}}(\overline{B}(x_j, j), \overline{B}(y_j, j)) \leq \frac{1}{j}$ but $\Eucl(X_j) \neq \Eucl(Y_j)$. By Proposition \ref{prop-GH-compactness-packing} the sequences $X_j$ and $Y_j$ converge in the pointed Gromov-Hausdorff sense to some spaces $X_\infty$ and $Y_\infty$ respectively, up to subsequences. The condition on the Gromov-Hausdorff distance of pointed balls says that $X_\infty = Y_\infty$. Theorem \ref{theo-intro-td} ensures that 
	$$\Eucl(Y_j) = \Eucl(Y_\infty) = \Eucl(X_\infty) = \Eucl(X_j)$$ 
	for $j$ big enough, giving the contradiction.
\end{proof}

\section{The proof of Theorem \ref{theo-intro-limit-whole-isometries}}
We conclude with the proof of Theorem \ref{theo-intro-limit-whole-isometries}.
\begin{proof}[Proof of Theorem \ref{theo-intro-stability}]
	There exist $P_0,r_0 > 0$ such that each $X_j$ is $(P_0,r_0)$-packed (see for instance the proof of \cite[Lemma 5.8]{Cav21ter}). Proposition \ref{prop-CM-decomposition} gives decompositions $X_j = M_j \times \R^{k_j} \times N_j$ with  $M_j$ symmetric Riemannian manifolds of non-positive curvature, $k_j = \Eucl(X_j)$ and $\Is(N_j)$ totally disconnected and unimodular. Since $\Is(X_j)$ is $D_0$-cocompact, so it is $\Is(N_j)$. Call $Y_{j} = \R^{k_{j}} \times N_{j}$. Let us take a subsequence $\lbrace j_h \rbrace$ for which $M_{j_h}$ converges in the pointed Gromov-Hausdorff sense to $M_\infty$. Since $X_{j_h}$ converges to $X_\infty$ then also $Y_{j_h}$ must converge to some space $Y_\infty$ and $X_\infty = M_\infty \times Y_\infty$. On $\R^{k_j}$ we consider a $1$-cocompact lattice $\Gamma_j$. The group $G_j := \Gamma_j \times \Is(N_j)$ is a closed, totally disconnected, unimodular, cocompact group of isometries of $Y_j$ with codiameter at most $\sqrt{1 +D_0^2}$. Theorem \ref{theo-intro-td} implies that 
	$$\Eucl(Y_\infty) = \lim_{j \to + \infty} \Eucl(Y_j) = \lim_{j \to + \infty} \Eucl(X_j).$$
	Since $\Eucl(X_\infty) = \Eucl(M_\infty) + \Eucl(Y_\infty)$ the thesis follows.
\end{proof}
	
	\bibliographystyle{alpha}
	\bibliography{Euclidean_factors}

\begin{thebibliography}{BdLHV08}

\bibitem[BdLHV08]{BdLHV08}
B.~Bekka, P.~de~La~Harpe, and A.~Valette.
\newblock Kazhdan's property (t).
\newblock {\em Kazhdan’s property (T)}, 11, 2008.

\bibitem[BH13]{BH09}
M.~Bridson and A.~Haefliger.
\newblock {\em Metric spaces of non-positive curvature}, volume 319.
\newblock Springer Science \& Business Media, 2013.

\bibitem[Cav21]{Cav21bis}
N.~Cavallucci.
\newblock Topological entropy of the geodesic flow of non-positively curved
  metric spaces.
\newblock {\em arXiv preprint arXiv:2105.11774}, 2021.

\bibitem[Cav22a]{Cav21ter}
N.~Cavallucci.
\newblock Continuity of critical exponent of quasiconvex-cocompact groups under
  gromov--hausdorff convergence.
\newblock {\em Ergodic Theory and Dynamical Systems}, pages 1--33, 2022.

\bibitem[Cav22b]{Cav21}
N.~Cavallucci.
\newblock Entropies of non-positively curved metric spaces.
\newblock {\em Geometriae Dedicata}, 216(5):54, 2022.

\bibitem[Cav23]{Cav23-GHTD}
N.~Cavallucci.
\newblock A gh-compactification of cat$(0)$-groups via totally disconnected,
  unimodular actions.
\newblock {\em arXiv preprint arXiv:2307.05640}, 2023.

\bibitem[CM09a]{CM09a}
P.E. Caprace and N.~Monod.
\newblock Isometry groups of non-positively curved spaces: discrete subgroups.
\newblock {\em J. Topol.}, 2(4):701--746, 2009.

\bibitem[CM09b]{CM09b}
P.E. Caprace and N.~Monod.
\newblock Isometry groups of non-positively curved spaces: structure theory.
\newblock {\em Journal of topology}, 2(4):661--700, 2009.

\bibitem[CM13]{CM13-unimodularity}
P.E. Caprace and N.~Monod.
\newblock Fixed points and amenability in non-positive curvature.
\newblock {\em Mathematische Annalen}, 356(4):1303--1337, 2013.

\bibitem[CM19]{CM19}
P.E. Caprace and N.~Monod.
\newblock Erratum and addenda to" isometry groups of non-positively curved
  spaces: discrete subgroups".
\newblock {\em arXiv preprint arXiv:1908.10216}, 2019.

\bibitem[CS20]{CavS20bis}
N.~Cavallucci and A.~Sambusetti.
\newblock Discrete groups of packed, non-positively curved, gromov hyperbolic
  metric spaces.
\newblock {\em arXiv preprint arXiv:2102.09829}, 2020.

\bibitem[CS21]{CavS20}
N.~Cavallucci and A.~Sambusetti.
\newblock Packing and doubling in metric spaces with curvature bounded above.
\newblock {\em Mathematische Zeitschrift}, pages 1--46, 2021.

\bibitem[CS23]{CS23}
N.~Cavallucci and A.~Sambusetti.
\newblock Convergence and collapsing of cat(0)-group actions.
\newblock {\em arXiv preprint arXiv:2304.10763}, 2023.

\bibitem[DK18]{DK18}
C.~Dru{\c{t}}u and M.~Kapovich.
\newblock {\em Geometric group theory}, volume~63.
\newblock American Mathematical Soc., 2018.

\bibitem[Far81]{farkas}
D.R. Farkas.
\newblock Crystallographic groups and their mathematics.
\newblock {\em Rocky Mountain J. Math.}, 11(4):511--551, 1981.

\bibitem[Fuk86]{Fuk86}
K.~Fukaya.
\newblock Theory of convergence for riemannian orbifolds.
\newblock {\em Japanese journal of mathematics. New series}, 12(1):121--160,
  1986.

\bibitem[Mor77]{Mor77}
S.A. Morris.
\newblock {\em Pontryagin Duality and the Structure of Locally Compact Abelian
  Groups}.
\newblock London Mathematical Society Lecture Note Series. Cambridge University
  Press, 1977.

\bibitem[NB65]{Nac65}
L.~Nachbin and L.~Bechtolsheim.
\newblock The haar integral.
\newblock {\em (No Title)}, 1965.

\end{thebibliography}
	
\end{document}